\documentclass[11pt]{amsart}
\usepackage{amsmath}
\usepackage{amsthm}
\usepackage{amssymb}
\usepackage{amsfonts,mathrsfs}
\usepackage{stmaryrd}
\usepackage{amsxtra}  
\usepackage{epsfig}
\usepackage{verbatim}
\usepackage[all]{xy}

\usepackage{mathtools}
\usepackage{tikz}

\theoremstyle{plain}
\newtheorem{theorem}[equation]{Theorem}

\newtheorem{lemma}[equation]{Lemma}
\newtheorem{corollary}[equation]{Corollary}

\theoremstyle{remark}
\newtheorem{remark}[equation]{Remark}
\numberwithin{equation}{section}

\newcommand{\dbar}{\bar \partial}

\def\norm#1{\Vert#1\Vert}

\newcommand{\ca}{{\mathcal A}}

\newcommand{\co}{{\mathcal O}}

\newcommand{\B}{{\mathbb B}}
\newcommand{\C}{{\mathbb C}}

\newcommand{\h}{{\mathbb H}}

\newcommand{\Z}{{\mathbb Z}}

\begin{document}

\title[Generalized Hartogs Triangles]{Bergman Theory of certain Generalized Hartogs Triangles}
\author{Luke D. Edholm}
\subjclass[2010]{32W05}
\begin{abstract}
The Bergman theory of domains $\{ |{z_{1} |^{\gamma}}  <  |{z_{2}} |  <  1 \}$ in $\C^2$ is studied for certain values of $\gamma$, including all positive integers.  For such $\gamma$, we obtain a closed form expression for the Bergman kernel, $\B_{\gamma}$.  With these formulas, we make new observations relating to the Lu Qi-Keng problem and analyze the boundary behavior of $\B_{\gamma}(z,z)$.
 \end{abstract}
\address{Department of Mathematics, \newline The Ohio State University, Columbus, Ohio, USA}
\email{edholm.1@osu.edu}

\maketitle 

\section{Introduction}

For a domain $\Omega \subset \C^n$, the Bergman space is the set of square-integrable, holomorphic functions on $\Omega$.  The Bergman kernel is a reproducing integral kernel on the Bergman space that is indispensable to the study holomorphic functions in several complex variables.  The purpose of this paper is to understand Bergman theory for a class of bounded, pseudoconvex domains in $\C^2$.  Define the {\em generalized Hartogs triangle of exponent} $\gamma > 0$ to be the domain
\begin{equation}\label{E:GenHartogsTriangle}
\h_{\gamma} = \{(z_1,z_2) \in \C^2: |{z_{1} |^{\gamma}}  <  |{z_{2}} |  <  1 \}.
\end{equation}
$\h_1$ is the ``classical'' Hartogs triangle, a well-known pseudoconvex domain with non-trivial Nebenh\"ulle.  When $\gamma > 1$, we call $\h_\gamma$ a {\em fat Hartogs triangle}, and when $0 < \gamma < 1$, we call $\h_\gamma$ a {\em thin Hartogs triangle}.  Our main results are the following two computations.

\begin{theorem}\label{T:fatHartogsformula}
 Let $s := z_{1}\bar{w}_{1}$, $t := z_{2}\bar{w}_{2}$, and $k \in \Z^+$.  The Bergman kernel for the fat Hartogs triangle $\h_{k}$ is given by
\begin{equation}\label{E:fatHartogsformula}
 \B_{k}(z,w) = \frac{p_{k}(s)t^{2}+q_{k}(s)t + s^{k}p_{k}(s)}{k\pi^{2}(1-t)^{2}(t -s^{k})^{2}}, 
\end{equation}
where $p_k$ and $q_k$ are the polynomials
\[
\\	p_{k}(s) = \sum_{l=1}^{k-1}l(k-l)s^{l-1},\qquad q_{k}(s) = \sum_{l=1}^{k}(l^2+(k-l)^2s^k)s^{l-1}.
\]
\end{theorem}

\begin{theorem}\label{T:thinHartogsformula}
  Let $s = z_{1}\bar{w}_{1}$, $t = z_{2}\bar{w}_{2}$, and $k \in \Z^+$.  The Bergman kernel for the thin Hartogs triangle $\h_{1/k}$ is given by
\begin{equation}\label{E:thinHartogsformula}
\B_{1/k}(z,w) = \frac{t^k}{\pi^2(1-t)^2(t^k-s)^2}.
\end{equation}
\end{theorem}

There has been an extensive amount of research devoted to understanding Bergman kernels of various classes of domains, and there are several instances in which explicit formulas for the kernel have been obtained. The most common method involves summing an infinite series, which is done in \cite{DAngelo78}, \cite{DAngelo94}, \cite{Park08}.  In \cite{BFS99}, explicit formulas for the Bergman kernel are produced using other techniques which avoid infinite series altogether.  But these situations are exceptional, and in most cases it is impossible to express the Bergman kernel in closed form.

Despite the difficulty of producing explicit formulas, powerful estimates on the Bergman kernel have been given for many classes of pseudoconvex domains.  In \cite{Fefferman74}, Fefferman develops an asymptotic expansion of the kernel on smoothly bounded, strongly pseudoconvex domains in $\C^n$.  Useful estimates also exist for large classes of smoothly bounded, weakly pseudoconvex domains.  See \cite{Catlin89}, \cite{McNeal89}, \cite{NagRosSteWai89}, \cite{McNeal94},  for some of the principal results on {\em finite type} domains, and \cite{Fu14} for domains with locally smooth boundaries and constant Levi-rank.

At present, there are no general theorems about the behavior of the Bergman kernel on pseudoconvex domains near unsmooth boundary points, which adds to the intrigue of Theorems \ref{T:fatHartogsformula} and \ref{T:thinHartogsformula}.  Each generalized Hartogs triangle defined by (\ref{E:GenHartogsTriangle}) has two very different kinds of boundary irregularities: the `corner points' which occur at the intersection of the two bounding real hypersurfaces, and the origin singularity, nearby which $b\h_k$ cannot be expressed as the graph of a continuous function.

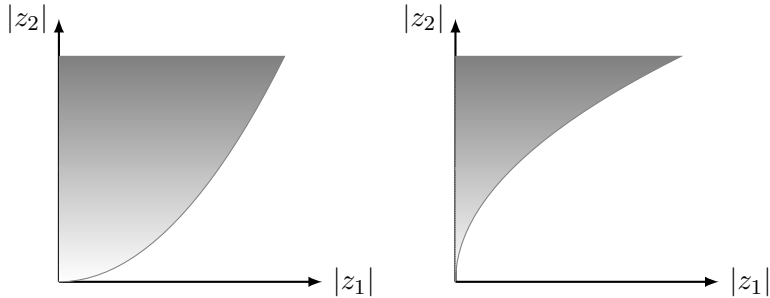
{\centering
\begin{tikzpicture}[x={(0,1cm)},y={(1 cm,0)}]

	\draw[-{latex}, thick] (0,0) -- (3.5,0) node[anchor=east] {$|z_2|$};
	\draw[-{latex}, thick] (0,0) -- (0,3.5) node[anchor=west] {$|z_1|$};
	\shadedraw [gray, domain=0:3] (0,0) -- plot  (1/3*\x*\x,\x) -- (3,0);

	\draw[-{latex}, thick, xshift=150] (0,0) -- (3.5,0) node[anchor=east] {$|z_2|$};
	\draw[-{latex}, thick, xshift=150] (0,0) -- (0,3.5) node[anchor=west] {$|z_1|$};
	\shadedraw [gray, domain=0:3, xshift=150] (0,0) -- plot  (\x,1/3*\x*\x) -- (3,0);
\end{tikzpicture} \\
}

This is one of several recent papers to study holomorphic function theory on domains with similar kinds boundary singularities.  In \cite{ChakShaw13}, Chakrabarti and Shaw investigate the Sobolev regularity of the $\dbar$-equation on the classical Hartogs triangle.  In \cite{ChaZey14}, Chakrabarti and Zeytuncu study the $L^p$-mapping properties of the Bergman projection on the classical Hartogs triangle, and in \cite{Chen14}, Chen studies $L^p$-mapping of the Bergman projection on analogous domains in higher dimensions.  Zapalowski, \cite{Zapa16}, characterizes proper maps between generalizations of the Hartogs triangle in $\C^n$.  This author and McNeal investigate the Bergman projection on fat Hartogs triangles in \cite{EdhMcN15}.  It can be hoped that by understanding the Bergman theory on example domains with boundary singularities such as $\h_{\gamma}$, we can gain deeper insight into the situation on more general domains.

\subsection{Acknowledgements} I would like to thank my advisor, Jeff McNeal, for introducing me to Bergman theory, and for his ongoing support and encouragement.  The many interesting discussions in his office have contributed to my growth, both as a mathematician and a human being.  I would also like to thank the anonymous referee who offered many useful suggestions to improve the original draft of this paper.

\bigskip
\bigskip


\section{Preliminaries}\label{S:prelims}

\subsection{Bergman theory}

Here we highlight some basic facts about Bergman theory that are used throughout this paper.  See \cite{Krantz_scv_book} for a more detailed treatment.  If $\Omega\subset\C^n$ is a domain, let $\co(\Omega)$ denote the holomorphic functions on $\Omega$. The standard $L^2$ inner product will be denoted

\begin{equation}\label{D:innerProduct}
\left\langle f,g\right\rangle =\int_\Omega f\cdot \bar g\, dV,
\end{equation}
where $dV$ denotes Lebesgue measure on $\C^n$. $L^2(\Omega)$ denotes the measurable functions $f$ such that $\langle f,f\rangle =\|f\|^2 <\infty$.  We define the {\em Bergman space} $A^2(\Omega):=\co(\Omega)\cap L^2(\Omega)$.

$A^2(\Omega)$ is a Hilbert space with inner product (\ref{D:innerProduct}), and for all $z \in \Omega$, the evaluation functional $\text{ev}_z: f \mapsto f(z)$ is continuous.  Therefore, the Riesz representation theorem guarantees the existence of a function $\B_{\Omega}:\Omega \times \Omega \to \C$ satisfying
\begin{equation}\label{E:ReproducingProp}
f(z) = \int_{\Omega} \B_{\Omega}(z,w)f(w) \,dV(w),\qquad f\in A^2(\Omega).
\end{equation}

We call $\B_{\Omega}$ the {\em Bergman kernel}, and when context is clear we may omit the subscript.  In addition to reproducing functions in the Bergman space via equation (\ref{E:ReproducingProp}), the Bergman kernel is conjugate symmetric and for each fixed $w \in \Omega$, $\B(\cdot \, ,w) \in A^2(\Omega)$.

Given an orthonormal Hilbert space basis \{$\phi_{\alpha}\}_{\alpha \in \ca}$ for $A^2(\Omega)$, the Bergman kernel is given by the following formula, which is independent of the choice of the basis:

\begin{equation}\label{E:BergmanInfiniteSum}
\B(z,w) = \sum_{\alpha \in \ca}\phi_{\alpha}(z)\overline{\phi_{\alpha}(w)}.
\end{equation}

Finally, the Bergman kernel transforms under biholomorphisms in the following way:  Let $F : \Omega \to \widetilde{\Omega}$ be a biholomorphic map of domains in $\C^n$.  Then
\begin{equation}\label{E:BiholoInvariance}
\B_{\Omega}(z,w) = \det{F'(z)} \cdot \B_{\tilde\Omega}(F(z),F(w)) \cdot \overline{\det{F'(w)}}.
\end{equation}

\bigskip

\subsection{The Bergman kernel of $\h_1$}\label{SS:h_1Formula}

The formula for the Bergman kernel of the classical Hartogs triangle has been known for quite some time, at least since Bremermann's paper \cite{Brem55}, in 1955.  Following the spirit of Bremermann's argument, we use formula (\ref{E:BiholoInvariance}) to compute $\B_{\h_1}$.  The map given by $\psi(z_1,z_2) = (\frac{z_1}{z_2},z_2)$ is a biholomorphism of $\h_1$ onto $D \times D^*$, where $D$ is the unit disc and $D^*$ is the punctured disc.  It's easy to see that the Bergman kernel of $D \times D^*$ is the same as that of the $D \times D$, which is well known and given by
\begin{equation}\label{E:BKBiDisc}
\B_{D \times D}(z,w) = \frac{1}{\pi^2 (1-z_1\bar{w}_1)^2 (1-z_2\bar{w}_2)^2} = \B_{D \times D^*}(z,w).
\end{equation}

Seeing that $\det{\psi'(z)} = \frac{1}{z_2}$, equation (\ref{E:BiholoInvariance}) says

\begin{equation}\label{E:BKClassicalHartogsTriangle}
\B_{\h_1}(z,w) = \frac{z_2\bar{w}_2}{\pi^2 (1-z_1\bar{w}_1)^2 (z_2\bar{w}_2-z_1\bar{w}_1)^2}.
\end{equation}

\bigskip

\subsection{Distance to the boundary and asymptotic growth rates}

The following notation will be used in section \ref{SS:DiagBoundaryBehav}.  Given any $z \in \Omega$, define the distance to the boundary of $\Omega$ function by 
\begin{equation*}
\delta_{\Omega}(z) := \min\{\norm{z-\zeta}:\zeta \in b\Omega\},
\end{equation*}  
where $\norm{\cdot}$ denotes Euclidean distance.  When the context is clear, we may omit the subscript.  We will also use the following notation to write inequalities.  If $A$ and $B$ are functions depending on several variables, write $A\lesssim B$ to mean that there is a constant $K>0$, independent of relevant variables, such that $A\leq K\cdot B$. The independence of which variables will be clear in context. Also write $A\approx B$ to mean that $A\lesssim B\lesssim A$.

\bigskip

\bigskip


\section{Bell's Transformation Rule and Derivation of the Kernel}

Equation (\ref{E:BiholoInvariance}) says that the Bergman kernels of two biholomorphic domains are related by a simple formula.  But applications of this transformation rule remain limited by the fact that it's rare to expect two domains in $\C^n$ to be biholomorphic.  There is, however, a more general version of this transformation rule.  In \cite{Bell82c}, Bell proves a generalization which applies whenever we have two domains and a proper holomorphic map from one onto the other.  The statement of this more general transformation rule appears below, and it will be essential to our proof of Theorem \ref{T:fatHartogsformula}.  

We first recall the classical fact that any holomorphic, proper map of $\Omega$ onto $\tilde{\Omega}$ is necessarily a branched covering of finite order.

\begin{theorem}[Bell's transformation rule, \cite{Bell82c}]\label{T:BellTransformation}
Let $\Omega$ and $\tilde{\Omega}$ be domains in $\C^n$ with respective Bergman kernels $\B$ and $\tilde{\B}$, and suppose $\phi$ is a proper holomorphic map of order $k$ from $\Omega$ onto $\tilde{\Omega}$.  Let $u := \det[\phi']$, and let  $\Phi_1, \Phi_2, \cdots, \Phi_k$ be the branch inverses of $\phi$ defined locally on $\tilde{\Omega} - V$, where $V := \{\phi(z): u(z) = 0\}$.  Finally, write $U_j := \det[\Phi_j']$.  Then,
\begin{equation}\label{E:BellFormula}
 u(z)\tilde{\B}(\phi(z),w) = \sum_{j=1}^k \B(z,\Phi_j(w))\overline{U_j(w)}.
\end{equation}
\end{theorem}

We're now ready to compute the Bergman kernel of fat Hartogs triangles with integer exponents.

\subsection{Proof of Theorem \ref{T:fatHartogsformula}}
For the rest of this paper, we'll denote the Bergman kernel of $\h_{\gamma}$ by $\B_{\gamma}$.  

\begin{proof} First we need to define the map $\phi$ and it's local inverses $\Phi_1, \cdots, \Phi_k$.  For each integer $k > 1$, $\phi:\h_1 \to \h_k$ given by $\phi(z) = (z_1, z_2^k) := (\phi_1(z),\phi_2(z))$ is a branch covering of order $k$, since
\begin{align*}
\{|\phi_1(z)|^k < |\phi_2(z)| < 1 \} &\Longleftrightarrow \{|z_1|^k < |z_2^k| < 1 \}\\
&\Longleftrightarrow \{|z_1| < |z_2| < 1 \}.
\end{align*}

We note $u(z) = kz_2^{k-1}$, so $V$ is the set $\{z_2 = 0\}$, which is disjoint from $\h_k$.  For each $j = 1, \dots, k$, the map $\Phi_j(z) = (z_1, \zeta^j z_2^{1/k})$ defines a local inverse of $\phi$, where $\zeta = e^{2\pi i/k}$ and $z_2^{1/k}$ is taken to mean the root with argument in the interval $[0, \frac{2\pi}{k})$. From this we see $U_j(z) = \zeta^j z_2^{1/k-1}$.  We now apply Bell's rule (\ref{E:BellFormula}):

\begin{align}
\B_k((z_1,z_2^k),(w_1,w_2)) &= \frac{z_2\bar{w}^{1/k}_2}{k^2 z_2^k\bar{w}_2} \sum_{j=1}^k \B_1((z_1,z_2),(w_1,\zeta^j w_2^{1/k})) \bar{\zeta}^j \notag \\
&= \frac{z_2^2\bar{w}^{2/k}_2}{\pi^2 k^2 z_2^k\bar{w}_2} \sum_{j=1}^k \frac{\bar{\zeta}^{2j}}{(1-z_2 \bar{w}^{1/k}_2 \bar{\zeta}^j)^2 (z_2 \bar{w}^{1/k}_2 \bar{\zeta}^j - z_1\bar{w}_1)^2} \notag \\
&= \label{E:Berg1} \frac{a^{2-k}}{\pi^2 k^2} \sum_{j=1}^k \frac{\bar{\zeta}^{2j}}{(1-a\bar{\zeta}^j)^2 (a\bar{\zeta}^j - s)^2},
\end{align}
where $a=z_2\bar{w}^{1/k}_2$ and $s=z_1\bar{w}_1$.  Define $f_j(a,s) := ({\zeta}^j-a)^2 (a - s{\zeta}^j)^2$ and notice that $\prod_{j=1}^k f_j(a,s) = \prod_{j=1}^k ({\zeta}^j-a)^2 \cdot \prod_{j=1}^k (a - s{\zeta}^j)^2 = (1-a^k)^2(a^k-s^k)^2$. 
 
Now, it follows that
\begin{align}
(\ref{E:Berg1}) &= \frac{a^{2-k}}{\pi^2 k^2} \sum_{j=1}^k \frac{{\zeta}^{2j}}{f_j(a,s)} \notag \\ 
&= \label{E:Berg2} \frac{a^{2-k} \sum_{j=1}^k F_j(a,s)\zeta^{2j}}{\pi^2 k^2(1-a^k)^2(a^k-s^k)^2},
\end{align}
where $F_j(a,s) := \frac{(1-a^k)^2(a^k-s^k)^2}{f_j(a,s)}$.  Notice each $F_j(a,s)$ can be written as a polynomial in $a$ of degree $4k-4$, so the numerator of (\ref{E:Berg2}) takes the following form:

\begin{equation}\label{E:Berg4}
a^{2-k} \sum_{j=1}^k F_j(a,s)\zeta^{2j} = \sum_{j=2-k}^{3k-2} g_j(s)a^j := G(a,s).
\end{equation}

We now wish to calculate the coefficient polynomials $g_j(s)$.  Toward this goal, observe that $G(\zeta^m a,s) = G(a,s)$ for all $m \in \Z$.  This follows because
\begin{align*}
G(\zeta^m a,s) &= (\zeta^m a)^{2-k} \sum_{j=1}^k F_j(\zeta^m a,s)\zeta^{2j}\\
&= a^{2-k} \sum_{j=1}^k \frac{(1-a^k)^2(a^k-s^k)^2}{f_{j-m}(a,s)} \zeta^{2j-2m} = G(a,s).
\end{align*}
Here, we've used the facts that $\label{E:Berg3}f_j(\zeta^m a,s) = \zeta^{4m} f_{j-m}(a,s)$ and $f_j(a,s) = f_{j+mk}(a,s)$ for all $m \in \Z$.  Because $G$ has this invariance, we conclude that
\begin{equation}\label{E:FatNumerator}
G(a,s) = a^{2-k} \sum_{j=1}^k F_j(a,s)\zeta^{2j} = g_{2k}(s)a^{2k}+g_{k}(s)a^{k}+g_{0}(s).
\end{equation}

It remains to calculate $g_{2k}(s), g_{k}(s)$ and $g_{0}(s)$, and these polynomials are obtained in the following lemma.  But to avoid disrupting the flow of the paper with several pages of algebra, we postpone its proof until section \ref{S:PolynomialComps}.

\begin{lemma}\label{L:PolynomialComps} The coefficient polynomials $g_{2k}(s), g_{k}(s)$ and $g_{0}(s)$ are given by the following formulas:
\begin{align}
\label{E:PolyComp1} &g_{2k}(s) = k\sum_{l=1}^{k-1}l(k-l)s^{l-1} := kp_k(s), \\
\label{E:PolyComp2} &g_{k}(s) = k\sum_{l=1}^{k}(l^2+(k-l)^2s^k)s^{l-1} := kq_k(s),\\
\label{E:PolyComp3} &g_0(s) = k\sum_{l=1}^{k-1}l(k-l)s^{k+l-1} = ks^kp_k(s).
\end{align}
\end{lemma}

Using this lemma and letting $t := a^k = z_2^k \bar{w}_2$, we see from (\ref{E:Berg2}) that
\begin{equation*}
\B_k((z_1,z_2^k),(w_1,w_2)) = \frac{p_k(s)t^2+q_k(s)+s^kp_k(s)}{k\pi^2(1-t)^2(t-s^k)^2}.
\end{equation*}
This is the desired formula for $\B_k$, except that both sides are a function of $z_2^k$.  This is remedied by formally replacing the variable $z_2^k$ with $z_2$.
\end{proof}

\begin{remark}
It's also true that the Bergman kernel of $\h_{m/n}$ is a rational function whenever $m,n \in \Z^+$.  Indeed, the map $(z_1,z_2) \mapsto (z_1z_2^{n-1},z_2^m)$ is a proper map from $\h_1$ onto $\h_{m/n}$, so Bell's formula gives $\B_{m/n}$ as a finite sum.  In \cite{Zapa16}, Zapalowski characterizes the proper maps between fat Hartogs triangles.  He shows there is a proper map $F:\h_{m/n} \to \h_{p/q}$ if and only if there are $a,b \in \Z^+$ such that
\begin{equation*}
\frac{aq}{p} - \frac{bn}{m} \in \Z.
\end{equation*}

Zapalowski's description of proper maps shows that the methods employed in this paper aren't able to say anything about fat Hartogs triangles $\h_{\gamma}$, for irrational $\gamma$. 
\end{remark}

\begin{remark}\label{R:Ramadanov}
Ramadanov's theorem says that if $\{\Omega_k\}$ is an increasing family of domains such that $\Omega_k \to \Omega \subset \subset \C^n$, then $\B_{\Omega_k}(z,w) \to \B_{\Omega}(z,w)$ absolutely and uniformly on compact subsets of $\Omega \times \Omega$.  See \cite{Ramadanov67} for the first appearance of the fact, and \cite{Boas96} for a generalization in the smoothly bounded, pseudoconvex case.  Notice that $\{ \h_k \}$ is an increasing family and that $\h_k \to D \times D^*$ as $k \to \infty$.  Ramadanov's theorem shows that $\B_k(z,w) \to \B_{D \times D^*}(z,w)$, which is given in equation (\ref{E:BKBiDisc}).  This is difficult to see from direct computation.
\end{remark}

\subsection{Biholomorphism classes of domains}

Let $\psi(z) = (\psi_1(z),\psi_2(z)) := (\frac{z_1}{z_2}, z_2)$.  In section \ref{SS:h_1Formula}, we used the fact that $\psi: \h_1 \to D \times D^*$ is a biholomorphism to compute the Bergman kernel of $\h_1$.  We'll give a very similar argument to prove Theorem \ref{T:thinHartogsformula}.  Let $\Psi(z) = (z_1z_2,z_2)$, and see that $\Psi: D \times D^* \to \h_1$ is the inverse of $\psi$.  Now, notice for all $k \in \Z^+$, $\psi : \h_{1/(k+1)} \to \h_{1/k}$ is also a biholomorphism, because
\begin{align*}
\{|\psi_1(z)|^{1/k} < |\psi_2(z)| < 1 \} &\Longleftrightarrow \{|\psi_1(z)| < |\psi_2(z)|^k < 1 \}\\
&\Longleftrightarrow \bigg\{\bigg|\frac{z_1}{z_2}\bigg| < |z_2|^k < 1 \bigg\}\\
&\Longleftrightarrow \{|z_1| < |z_2|^{k+1} < 1 \}\\
&\Longleftrightarrow \{|z_1|^{1/(k+1)} < |z_2| < 1 \}.
\end{align*}

Let $\psi^k := \psi \circ \cdots \circ \psi$ be $k$ copies of $\psi$ composed together, so $\psi^k(z) := (z_1z_2^{-k}, z_2)$.  This gives a biholomorphism from $\h_{1/k} \to D \times D^*$ with inverse $\Psi^k := \big(z_1z_2^k,z_2\big)$.  We illustrate this chain of biholomorphisms below.

\begin{equation*}
\begin{matrix}
&\Psi&&\Psi&&\Psi \qquad \Psi&&\Psi&\\
D\times D^*&\rightleftharpoons &\h_1&\rightleftharpoons&\h_{1/2}&\rightleftharpoons \cdots \rightleftharpoons &\h_{1/k} &\rightleftharpoons &\cdots \\
&\psi&&\psi&&\psi \qquad \psi&& \psi &
 
\end{matrix}
\end{equation*}

\subsection{Proof of Theorem \ref{T:thinHartogsformula}}  Using the biholomophism $\psi^k:\h_{1/k} \to D \times D^*$ we can easily prove the desired formula.

\begin{proof} Since $\det [\psi^k]'(z) = z_2^{-k}$,
\begin{align*}
\B_{1/k}(z,w) = \frac{1}{z_2^{k}\bar{w}_2^{k}} \B_1(\psi_k(z),\psi_k(w))
= \frac{z_2^k\bar{w}_2^k}{\pi^2 (1-z_1\bar{w}_1)^2 (z_2^k\bar{w}^k_2-z_1\bar{w}_1)^2}.
\end{align*}
\end{proof}

\begin{remark}
For $m,n \in  \Z^+$, the map $\psi(z) = (\frac{z_1}{z_2}, z_2)$ also gives a biholomorphism from $\h_{m/(n+m)}$ onto $\h_{m/n}$.  Applying this map recursively, we see that $\h_{m/(n+km)}$ and $\h_{m/n}$ and biholomorphic for all $k \in \Z^+$.
\end{remark}


\section{Consequences of the Kernel Formulas}

\subsection{The Lu Qi-Keng Problem}

One of the long-standing open problems in Bergman theory is to classify the domains  for which the Bergman kernel is nowhere vanishing.  This question was first raised by Lu Qi-Keng in \cite{LQK66}.  We say that a domain $\Omega \subset \C^n$ is Lu Qi-Keng when it has zero-free Bergman kernel, and the investigation of which domains have a zero-free Bergman kernel is known as the Lu Qi-Keng problem.  See \cite{Boas00} for a good historical survey, a few key points of which we now summarize.

The situation in the complex plane is relatively straightforward.  When $\Omega \subset \C$ is simply connected, the Riemann mapping theorem together with equation (\ref{E:BiholoInvariance}) show that $\Omega$ is a Lu Qi-Keng domain, since the Bergman kernel of the unit disc is non-vanishing.  But a finitely-connected domain in $\C$ with at least two non-singleton boundary components is not Lu Qi-Keng.  See \cite{Rosenthal69} and \cite{Skwar69} when $\Omega$ is an annulus, and \cite{Bell_CTPTCM} for a more general class of domains.  

There is no such simple characterization of the situation known in higher dimensions.  In \cite{BFS99}, it's shown there are smoothly bounded, strongly convex domains with real analytic boundary that are not Lu Qi-Keng in $\C^n$, when $n \ge 3$.  Contrary to previous expectations, Boas shows in \cite{Boas96} that `most' pseudoconvex domains (with respect to a certain topology on the set of domains in $\C^n$) have vanishing Bergman kernel.  Nevertheless, it is still desirable to understand why domains from certain classes have zero-free Bergman kernels, while domains from closely related classes may not.  We now address this problem in the case of the domains $\h_{\gamma}$, where $\gamma \in \Z^+$ and $\gamma^{-1} \in \Z^+$.

Using the explicit formulas for the Bergman kernels computed in the previous section, we can check whether or not these domains are Lu Qi-Keng.  The following corollary is immediate from equation (\ref{E:thinHartogsformula}), whose numerator vanishes if and only if at least one of $z_2$ or $w_2$ equals zero.

\begin{corollary}\label{C:LQKthinHartogs}
Let $k$ be a positive integer.  The thin Hartogs triangle $\h_{1/k}$ is a Lu Qi-Keng domain.
\end{corollary}

For fat Hartogs triangles with integer exponent $k \ge 2$, we deduce the following corollary from equation (\ref{E:fatHartogsformula}).

\begin{corollary}\label{C:LQKfatHartogs}
Let $k \ge 2$ be an integer. The fat Hartogs triangle $\h_k$ is not a Lu Qi-Keng domain.
\end{corollary}

\begin{proof}  

First consider the case where $k \ge 3$.  Let $z = (0,\frac{i}{\sqrt{k-1}})$ and $w = (0,\frac{-i}{\sqrt{k-1}})$.  Then $z,w \in \h_k$.  Since $p_k(0) = k-1$ and $q_k(0) = 1$, we see that $\B_k(z,w) = 0$.  When $k=2$, let $z = (\frac{i}{\sqrt{2}},\frac{\sqrt{7} + i}{4})$ and $w = (\frac{-i}{\sqrt{2}},\frac{\sqrt{7} - i}{4})$.  It is easily checked that $z,w \in \Omega_2$ and that $\B_2(z,w) = 0$.
\end{proof}

It's immediate from equation (\ref{E:BiholoInvariance}) that a non-vanishing Bergman kernel is a biholomorphic invariant.  Corollary \ref{C:LQKfatHartogs} lets us deduce the following:

\begin{corollary}\label{C:biholoInequiv}  Let $k \ge 2$ be an integer.  $\h_k$ is not biholomorphic to $D \times D^*$.
\end{corollary}

\begin{remark}\label{R:nonintegerLQK}
Using Ramadanov's theorem in conjunction with Hurwitz's theorem on zeroes of holomorphic functions, we see that for each integer $k \ge 2$, there is an $s_k \in [k-1,k)$ such that for all $\gamma \in (s_k,k]$, the Bergman kernel $\B_{\gamma}$ of $\h_{\gamma}$ has zeroes.  It seems plausible to conjecture that $s_k = k-1$, i.e., that no fat Hartogs triangle of exponent $\gamma > 1$ is Lu Qi-Keng.  
\end{remark}

\begin{remark}\label{R:LQKLimit}
As was mentioned in Remark \ref{R:Ramadanov}, $\h_k \to D \times D^*$ as $k \to \infty$.  The Bergman kernel $\B_{D \times D^*}$ is zero free, so for any fixed compact subset $K \subset D \times D^*$, Ramadanov's theorem tells us that the Bergman kernel $\B_k$ restricted to $K$ is zero free for all $k$ sufficiently large.  We see this happen as the zero of $\B_k$ provided in the proof of Corollary \ref{C:LQKfatHartogs} is pushed to the origin.  It would be interesting to do further analysis of the zero set for the polynomial in the numerator of $\B_k$.
\end{remark}

\subsection{Diagonal boundary behavior}\label{SS:DiagBoundaryBehav}

The asymptotic behavior of $\B_{\Omega}(z,z)$ as $z$ tends to the boundary has been studied for many classes of smoothly bounded, pseudoconvex domains.  \cite{HormanderActa65} and  \cite{Fefferman74} are two seminal papers dealing with the strongly pseudoconvex case.  Results also exist for many classes of smoothly bounded, weakly pseudoconvex domains.  See \cite{McNeal89}, \cite{Catlin89}, \cite{NagRosSteWai89} for finite-type domains in $\C^2$, and \cite{McNeal94} for finite-type, convex domains in $\C^n$.  Refer to \cite{Fu14} for analogous results on smoothly bounded domains with constant Levi rank.  But all these estimates are for classes of domains with boundary smoothness, and there are presently no general theorems about the behavior of $\B_{\Omega}(z,z)$ for pseudoconvex domains near singular boundary points.

Using the explicit formulas for the Bergman kernel, we establish the following lemma.

\begin{lemma}\label{L:fatHartogsformulaApprox}
Let $k \in \Z^+$.  Then we have the following asympotic behavior of the Bergman kernel restricted to the diagonal:
\begin{align}
&\B_k(z,z) \approx \frac{1}{(1-|z_2|)^2(|z_2|-|z_1|^k)^2}, \qquad z \in \h_k. \label{E:fatHartogsformulaApprox}
\end{align}
\end{lemma}

\begin{proof}
In this proof we are concerned with $\B_k(z,z)$, so write $s := |z_1|^2$ and $t := |z_2|^2$.  From Theorem \ref{T:fatHartogsformula} we see that 
\begin{align}
\B_k(z,z) = \frac{p_{k}(s)t^{2}+q_{k}(s)t + s^{k}p_{k}(s)}{k\pi^{2}(1-t)^{2}(t -s^{k})^{2}} \label{E:fatHartogsformulaDiag},
\end{align}
where $p_k(s)$ and $q_k(s)$ are given in the statement of Theorem \ref{T:fatHartogsformula}.  We now estimate the numerator of (\ref{E:fatHartogsformulaDiag}).  Notice that $q_k(s) \ge 1$ for all $s \in [0,1)$, and so
\begin{align*}
t &\le p_k(s)t^2 + q_k(s)t + s^kp_k(s) \\
& < t[2p_k(1) + q_k(1)] \\
& \lesssim t,
\end{align*}
since $s^k < t$.  Now estimate the terms in the denominator.  It's easy to see that both
\begin{align*}
&(1-t)^2 \approx  (1-|z_2|)^2,  \\
&(t-s^k)^2 \approx  |z_2|^2(|z_2|-|z_1|^k)^2.
\end{align*}
Here, we've used the fact that $|z_2|^2\le(|z_2|+|z_1|^k)^2<4|z_2|^2$. Putting these estimates together, we obtain (\ref{E:fatHartogsformulaApprox}).  
\end{proof}

Let $\Omega \subset \C^2$ be a bounded domain and $\zeta \in b\Omega$ a smooth, Levi-flat boundary point.  It can be shown that $\B_{\Omega}(z,z) \approx \delta_{\Omega}(z)^{-2}$ as $z \to \zeta$.  See \cite{Fu14} for more information.  The domains $\h_{k}$ are Levi-flat at all smooth boundary points, because the smooth parts of the boundary can be locally foliated by analytic discs.  We explicitly see this asymptotic behavior from estimate (\ref{E:fatHartogsformulaApprox}).  In fact, this estimate also lets us determine the asymptotic growth rate of $\B_{k}(z,z)$ as $z$ tends to the boundary singularity at the origin.  When $z$ is sufficiently close to $0$, it's straightforward to see $|z_2|-|z_1|^k \approx \delta_k(z)$, the distance of $z$ to the boundary of $\h_k$.  From this, we deduce

\begin{theorem}
Let $k \in \Z^+$ and $\delta_k(z)$ be the distance of $z$ to $b\h_k$.  Then 
\begin{align*}
 \B_k(z,z) \approx \delta_k(z)^{-2} \qquad as \quad z \to 0.
\end{align*}
\end{theorem}

\begin{remark}
Following steps analogous to those in Lemma \ref{L:fatHartogsformulaApprox}, we can show
\begin{align*}
\B_{1/k}(z,z) \approx \frac{1}{(1-|z_2|)^2(|z_2|^k-|z_1|)^2}, \qquad z \in \h_{1/k}. \label{E:thinHartogsformulaApprox}
\end{align*}
This estimate can be used to determine the asymptotic growth rate of $\B_{1/k}(z,z)$ as $z$ tends to boundary singularity at the origin.  When $z$ is sufficiently close to $0$, it's straightforward to check that $|z_2|^k-|z_1| \approx \delta_{1/k}(z)$, the distance of $z$ to $b\h_{1/k}$.  From this we conclude that $\B_{1/k}(z,z) \approx \delta_{1/k}(z)^{-2}$ as $z \to 0$.
\end{remark}


\section{Proof of Lemma \ref{L:PolynomialComps}}\label{S:PolynomialComps}
Equation (\ref{E:FatNumerator}) tells us that 
\begin{equation}
a^{2-k} \sum_{j=1}^k F_j(a,s)\zeta^{2j} = g_{2k}(s)a^{2k}+g_{k}(s)a^{k}+g_{0}(s).
\end{equation}
We'll prove Lemma \ref{L:PolynomialComps} by splitting the calculation of $g_{2k}(s)$, $g_{k}(s)$ and $g_{0}(s)$ into two separate lemmas.

\begin{lemma}\label{L:CoeffComp1}
Let $h_l(s) := \sum_{r=0}^l s^r$.  For each $j=1, \dots, k$, the respective coefficient functions of the $a^{3k-2}, a^{2k-2}$ and $a^{k-2}$ terms of $F_j(a,s)\zeta^{2j}$ are equal to the following:
\begin{align*}
&a^{3k-2}: \qquad \sum_{l=0}^{k-2}h_l(s)h_{k-2-l}(s),\\
&a^{2k-2}: \qquad 2\sum_{l=0}^{k-2} s^{k-1-l} h_l(s)^2 + h_{k-1}(s)^2,\\
&a^{k-2}: \qquad s^k \sum_{l=0}^{k-2} h_l(s) h_{k-2-l}(s).
\end{align*}
In particular, note that these expressions have no $j$ dependence.
\end{lemma}

\begin{proof}
In this calculation of the coefficient functions of the $a^{3k-2}, a^{2k-2}$ and $a^{k-2}$ terms appearing in $F_j(a,s)\zeta^{2j}$, we'll often write $\theta := \zeta^j$ to cut down on superscripts.

\begin{align}
F_j(a,s) &= \frac{(1-a^k)^2(a^k-s^k)^2}{f_j(a,s)} = \bigg(\frac{a^k-1}{a-\theta}\bigg)^2 \bigg(\frac{a^k-s^k}{a-s\theta}\bigg)^2 \notag \\
&= \bigg( \sum_{m=1}^k a^{k-m} \theta^{m-1} \bigg)^2\bigg( \sum_{n=1}^k a^{k-n} (s\theta)^{n-1} \bigg)^2 \notag \\
&= \bigg( \sum_{m=1}^k \sum_{n=1}^k a^{2k-m-n} \theta^{m+n-2} s^{n-1} \bigg)^2. \label{E:DoubleSum}
\end{align}

To better understand the double sum {\em inside} the parentheses of (\ref{E:DoubleSum}) above, we split this sum into three pieces, $A, B$ and $C$, depending on the value of $m+n$.  Let $A$ be the sum of the terms with $2 \le m+n \le k$, $B$ be the sum of the terms with $m+n=k+1$, and $C$ be the sum of the terms with $k+2 \le m+n \le 2k$.

We re-write $A$ by letting $l=m+n-2$ be the index of summation.  Then
\begin{align*}
A =\sum_{l=0}^{k-2}a^{2k-l-2} \theta^l h_l(s).
\end{align*}
For $B$, only include those terms with $m+n=k+1$, so we don't have an outside sum.  Therefore,
\begin{equation*}
B = a^{k-1} \theta^{k-1} h_{k-1}(s).
\end{equation*}
For $C$, let $l = m+n-k-2$ be the index of summation.  Then
\begin{align*}
C =\sum_{l=0}^{k-2} a^{k-2-l} \theta^{k+l} s^{l+1}h_{k-2-l}(s).
\end{align*}
So we have

\begin{align*}
(\ref{E:DoubleSum}) 
&=  \bigg(\sum_{l=0}^{k-2} a^{2k-2-l}\theta^l h_l(s) + a^{k-1}\theta^{k-1} h_{k-1}(s) + \sum_{l=0}^{k-2} a^{k-2-l}\theta^{k+l} s^{l+1} h_{k-2-l}(s)\bigg)^2  \notag \\
&= (A+B+C)^2 \notag \\
&= A^2 + B^2 + C^2 +2AB +2BC +2AC.
\end{align*}

I emphasize that as a polynomial in $a$, $A$ has powers of $a$ ranging from $a^{2k-2}$ to $a^k$, $B$ only has an $a^{k-1}$ term, and $C$ has terms ranging from $a^{k-2}$ to $a^0$.  This observation greatly simplifies the computations below.

\textbf{1. Computation of the $a^{3k-2}$ coefficient:}  For the coefficient of the $a^{3k-2}$ term in $F_j(a,s)\theta^2$, it is sufficient to consider the coefficient function of $a^{3k-2}$ in $A^2\theta^2$:
\begin{align*}
A^2 \theta^2 &= \bigg(\sum_{m=0}^{k-2} a^{2k-2-m}\theta^m h_m(s) \bigg) \bigg(\sum_{n=0}^{k-2} a^{2k-2-n}\theta^n h_n(s) \bigg)\theta^2 \\
&= \theta^2 \sum_{m=0}^{k-2} \sum_{n=0}^{k-2} a^{4k-4-m-n}\theta^{m+n} h_m(s)h_n(s).
\end{align*}
Letting $m+n=k-2$, we find the coefficient function of $a^{3k-2}$ is independent of $\theta$ (since $\theta^k = 1$), and therefore independent of $j$.  This function is given by
\begin{equation}
\sum_{l=0}^{k-2} h_l(s)h_{k-2-l}(s).\label{F:CoeffSum1}
\end{equation}

\textbf{2. Computation of the $a^{2k-2}$ coefficient:}  For the coefficient of the $a^{2k-2}$ term in $F_j(a,s)\theta^2$, it is sufficient to consider the coefficient of $a^{2k-2}$ in $(2AC + B^2)\theta^2$:
\begin{align*}
&(2AC + B^2)\theta^2 = \\
&=\bigg[2\big(\sum_{m=0}^{k-2} a^{2k-2-m}\theta^m h_m(s) \big) \big(\sum_{n=0}^{k-2} a^{k-2-n}\theta^{k+n} s^{n+1} h_{k-2-n}(s)\big) + (a^{k-1}\theta^{k-1}h_{k-1}(s))^2 \bigg] \theta^2 \\
&= \bigg[ 2\sum_{m=0}^{k-2} \sum_{n=0}^{k-2} a^{3k-4-m-n}\theta^{k+m+n} s^{n+1} h_m(s)h_{k-2-n}(s) + a^{2k-2}\theta^{2k-2}h_{k-1}(s)^2\bigg] \theta^2.
\end{align*}
Letting $m+n=k-2$, we find the coefficient function of $a^{2k-2}$ is independent of $\theta$ (since $\theta^{2k} = 1$), and therefore independent of $j$.  This function is given by
\begin{equation}
 2\sum_{l=0}^{k-2} s^{k-1-l} h_l(s)^2 + h_{k-1}(s)^2.\label{F:CoeffSum2}
\end{equation}

\textbf{3. Computation of the $a^{k-2}$ coefficient:}  For the coefficient of the $a^{k-2}$ term in $F_j(a,s)\theta^2$, it is sufficient to determine the coefficient of $a^{k-2}$ in $C^2\theta^2$:
\begin{align*}
C^2\theta^2&=\bigg(\sum_{m=0}^{k-2} a^{k-2-m}\theta^{k+m} s^{m+1} h_{k-2-m}(s)\bigg) \bigg(\sum_{n=0}^{k-2} a^{k-2-n}\theta^{k+n} s^{n+1} h_{k-2-n}(s)\bigg)\theta^2\\
&= \theta^2 \sum_{m=0}^{k-2}\sum_{n=0}^{k-2} a^{2k-4-m-n}\theta^{2k+m+n} s^{m+n+2} h_{k-2-m}(s)h_{k-2-n}(s).
\end{align*}
Letting $m+n=k-2$, we find the coefficient function of $a^{k-2}$ is independent of $\theta$ (since $\theta^{3k} = 1$), and therefore independent of $j$.  This function is given by
\begin{equation}
s^k \sum_{l=0}^{k-2} h_l(s) h_{k-2-l}(s).\label{F:CoeffSum3}
\end{equation}
\end{proof}

Now we re-express (\ref{F:CoeffSum1}), (\ref{F:CoeffSum2}) and (\ref{F:CoeffSum3}) as simpler polynomials:

\begin{lemma}\label{L:CoeffComp2} Again, let $h_l(s) = \sum_{r=0}^{l}s^r$.  Then we have the following equalities:
\begin{align}
\sum_{l=0}^{k-2} h_l(s)h_{k-2-l}(s) &= \sum_{l=1}^{k-1}l(k-l)s^{l-1}, \label{E:SumEquality1}\\
2\sum_{l=0}^{k-2} s^{k-1-l} h_l(s)^2 + h_{k-1}(s)^2 &= \sum_{l=1}^{k}(l^2+(k-l)^2s^k)s^{l-1}. \label{E:SumEquality2}
\end{align}
\end{lemma}

\begin{proof}
Focus on (\ref{E:SumEquality1}) first.  Notice that 
\begin{align*}
h_l(s)h_{k-2-l}(s) &= \bigg(\sum_{m=0}^l s^m\bigg)\bigg(\sum_{n=0}^{k-2-l} s^n\bigg)\\
&= \sum_{r=0}^{k-2}s^r + \sum_{r=1}^{k-3}s^r + \cdots +\sum_{r=L}^{k-2-L}s^r,
\end{align*}
where $L=\min\{l,k-2-l\}$.  Using this, we see
\begin{align}
(\ref{F:CoeffSum1}) &= \sum_{l=0}^{k-2} h_l(s)h_{k-2-l}(s) \notag \\ &= \sum_{l=0}^{k-2} \bigg( \sum_{r=0}^{k-2}s^r + \sum_{r=1}^{k-3}s^r + \cdots +\sum_{r=L}^{k-2-L}s^r \bigg) \notag \\
&= (k-1)\sum_{r=0}^{k-2}s^r + (k-3)\sum_{r=1}^{k-3}s^r + \cdots + (k-2K-1)\sum_{r=K}^{k-2-K}s^r, \notag\label{E:SumDecomp1}
\end{align}
where $K = \big\lfloor \frac{k-2}{2}\big\rfloor$.  From here, we compute that the coefficient of $s^l$ in (\ref{F:CoeffSum1}) is given by
\begin{align*}
\sum_{m=0}^{L}(k-2m-1) &= (L+1)(k-1) - 2\sum_{m=0}^L m \\&= (L+1)(k-L-1) \\&= (l+1)(k-l-1).
\end{align*}

Therefore,
\begin{align*}
(\ref{F:CoeffSum1}) &= \sum_{l=0}^{k-2} h_l(s)h_{k-2-l}(s)\\
&=\sum_{l=0}^{k-2} (l+1)(k-l-1) s^l = \sum_{l=1}^{k-1} l(k-l)s^{l-1},
\end{align*}
where we've re-indexed the sum in the last equality, obtaining the form of (\ref{E:SumEquality1}).

Now we'll establish (\ref{E:SumEquality2}).  Note that 
\begin{align*}
h_r(s)^2 &= 1 + 2s + \cdots + rs^{r-1} + (r+1)s^r + rs^{r+1} + \cdots + 2s^{2r-1} + s^{2r}.
\end{align*}
Using this, write the pieces of $(\ref{F:CoeffSum2}) = 2\sum_{l=0}^{k-2} s^{k-1-l} h_l(s)^2 + h_{k-1}(s)^2$ in the following way:

\begin{equation*}
\begin{matrix}
s^{k-1}h_0(s)^2 &=   & & & & s^{k-1} & & & & \\
s^{k-2}h_1(s)^2 &=   & & &s^{k-2} & +2s^{k-1} &+s^k& & & \\
s^{k-3}h_2(s)^2 &=   & &s^{k-3} &+2s^{k-2} & +3s^{k-1} &+2s^k&+s^{k+1} & & \\
\vdots & & & & \vdots&\vdots & \vdots & & &\\
h_{k-1}(s)^2 &=  &1 &+\cdots  &+(k-1)s^{k-2} & +ks^{k-1} &+(k-1)s^k& +\cdots &+s^{2k-2} 
\end{matrix}
\end{equation*}

The coefficient of $s^l$ in (\ref{F:CoeffSum2}) can be obtained by considering the vertical columns above.  Notice the coefficient of $s^l$ and $s^{2k-2-l}$ are always the same.  When $0 \le l \le k-1$, we have that the coefficient of $s^l$ is given by 
\begin{equation*}
2 \sum_{r=1}^{l}r + (l+1) =  (l+1)^2.
\end{equation*}

Therefore,
\begin{align*}
(\ref{F:CoeffSum2}) &= \sum_{l=0}^{k-1} (l+1)^2 s^l + \sum_{l=0}^{k-2} (k-(l+1))^2 s^{k+l}\\
&= \sum_{l=0}^{k-1} \big((l+1)^2 + (k-l-1)^2s^k \big)s^{l}\\
&= \sum_{l=1}^{k} \big(l^2 + (k-l)^2s^k \big)s^{l-1},
\end{align*}
where we have re-indexed the sum in the last equality to obtain the form of (\ref{E:SumEquality2}).
\end{proof}

\subsection{Proof of Lemma \ref{L:PolynomialComps}}
\begin{proof}
Putting Lemma \ref{L:CoeffComp1} together with Lemma \ref{L:CoeffComp2} gives us (\ref{E:PolyComp1}), (\ref{E:PolyComp2}) and (\ref{E:PolyComp3}).
\end{proof}

\bibliographystyle{acm}
\bibliography{Edh}

\end{document}